\documentclass[a4paper, 11pt]{amsart} %remove ``draft'' before final; add ``leqno'' for equation numbers to be written on the left; add ``twoside'' to suit it for books original: documentclass[a4paper,11pt]{article}

\usepackage{amsmath} %mandatory, math
\usepackage{verbatim} 
\usepackage{amsthm} %mandatory, theorem environment
\usepackage{amssymb} %mandatory, symbols
\usepackage{mathrsfs} %mandatory, fonts
\usepackage[all]{xy} 
\usepackage{pdfsync} %used for synchronization between a pdf file and the TeX

\numberwithin{equation}{section} %ads section number to equations
\theoremstyle{plain}
\newtheorem{theorem}{Theorem}[section]
\newtheorem{proposition}[theorem]{Proposition}
\newtheorem{lemma}[theorem]{Lemma}
\newtheorem{corollary}[theorem]{Corollary}

\theoremstyle{definition}
\newtheorem{definition}[theorem]{Definition}

\theoremstyle{remark}
\newtheorem{remark}[theorem]{Remark}

\newcommand{\id}{\operatorname{id}}
\newcommand{\Cs}{{$C^*$-algebra}}
\newcommand{\N}{\mathbb N}
\newcommand{\Z}{\mathbb Z}

\newcommand{\R}{\mathbb R}
\newcommand{\C}{\mathbb C}

\newcommand{\B}{\mathbb B}
\newcommand{\E}{\mathbb E}

\newcommand{\cP}{\mathcal P}
\newcommand{\cO}{\mathcal O}

\newcommand{\supp}{\operatorname{supp}}

\makeatletter \newcommand \listoftodos{\@starttoc{tdo}}
 \newcommand\l@todo[2]
   {\par\noindent \textit{#2}, \parbox{10cm}{#1}\par} \makeatother

\author{Mikael R\o rdam}
\address{Department of Mathematical Sciences, University of Copenhagen, Universitets\-parken 5, DK-2100, Copenhagen \O, Denmark} \email{rordam@math.ku.dk}
\author{Adam Sierakowski}
\address{The Fields Institute for Research in Mathematical Sciences,
  222 College Street, Second Floor, Toronto, M5T 3J1 Canada}
\address{Department of Mathematics and Statistics, York University, 4700
  Keele Street, Toronto, M3J 1P3 Canada}
\email{asierako@fields.utoronto.ca, adamdk@yorku.ca}
\date{\today} %empty argument in final
\title{Purely infinite $C^*$-algebras arising from crossed products}
\thanks{M.R.\ was supported by grants from the Danish National
  Research Foundation and the Danish Natural Science
  Research Council (FNU)}
\thanks{A.S.\ was supported by grants of Professors G.A.Elliott and
  A.S.Toms and from the Visitor Fund at the Department of Mathematics
  and Statistics, York University} 
\begin{document}
\maketitle

\begin{abstract}
We study conditions that will ensure that a
crossed product of a \Cs{} by a discrete exact group is purely
infinite (simple or non-simple). We are particularly interested in the
case of a discrete non-amenable exact group acting on a commutative
\Cs, where our sufficient conditions can be phrased in terms of
paradoxicality of subsets of the spectrum of the abelian \Cs. 

As an application of our results we show that every discrete countable
non-amenable exact group admits a free amenable minimal action on the
Cantor set such that the corresponding crossed product \Cs{} is a
Kirchberg algebra in the UCT class.
\end{abstract}

\section{Introduction}

Operator algebras and dynamical systems are closely related. Dynamical
systems give rise to operator algebras and they can be studied in
terms of these algebras and their invariant. It is a challenging task to
decipher information about the dynamical system from its corresponding
operator algebra and vice versa. 

One motivating example for this article is the
particular dynamical
system where an arbitrary (discrete) group $G$ acts on
$\ell^\infty(G)$ by left-translation. The associated (reduced) crossed
product \Cs{} $\ell^\infty(G) \rtimes_r G$, also known as the
\emph{Roe algebra}, plays an important role in the study of $K$-theory
of groups, but it also harbors information related to paradoxical sets
and Banach-Tarski's paradox. For example,
if a subset $E$ of $G$ is $G$-paradoxical, then $1_E \in
\ell^\infty(G)$ is a properly infinite projection in the Roe
algebra. (We show in this paper that the converse also holds.)
In particular, the Roe algebra itself is properly infinite if
and only if $G$ is $G$-paradoxical which happens if and only if
$G$ is non-amenable. 

The special class of \Cs s, now called \emph{Kirchberg algebras}, that
are purely infinite simple separable and nuclear, are of particular
interest because of their classification (by $K$- or $KK$-theory)
obtained by Kirchberg and Phillips in the mid 1990's. Many
of the naturally occurring examples of Kirchberg algebras arise from
dynamical systems. The Cuntz algebras $\cO_n$, for example,
are stably isomorphic to
the crossed product of a stabilized UHF-algebra with an action of the
group of integers that scales the trace. 

Several classes of examples of Kirchberg algebras arising as crossed
products of abelian \Cs s (often times with spectrum the Cantor set)
by hyperbolic groups have appeared in the literature. Prompted by
Choi's embedding of $C^*_r(\Z_2 * 
\Z_3)$ into $\cO_2$, Archbold and Kumjian (independently) proved that
there is an action of $\Z_2 * \Z_3$ on the Cantor set so that the
corresponding crossed product \Cs{} is isomorphic to $\cO_2$. 

It is well-known that if $G$ acts topologically 
freely and minimally on a compact
Hausdorff space $X$, then the crossed product $C(X) \rtimes_r G$ is
simple. If $G$ acts amenably on a (general) \Cs{} $A$, then $(A,G)$ is
automatically
\emph{regular} (meaning that the full and the reduced crossed products
are the same), and $A \rtimes G$ is nuclear if and only
if $A$ is nuclear, see \cite[Theorem 4.3.4]{BroOza}. Conversely, if
$A$ is abelian and unital and $A \rtimes_r G$ is nuclear, then the
action of $G$ on $A$ is amenable, see \cite[Theorem 4.4.3]{BroOza}.
For abelian \Cs s $A$, the full crossed
product $A \rtimes G$ 
is simple if and only if the action is minimal, topologically
free, and regular, \cite[Corollary, p.\ 124]{ArcSpi}. These results
can be reformulated as follows,
assuming $A$ to be unital and abelian: the reduced crossed product $A
\rtimes_r G$ is simple \emph{and} nuclear if and only if the action is
minimal, topologically free, and amenable. See \cite{Del3, Del2,
  ArcSpi,BroOza} for details. 

There are several
known partial results containing sufficient conditions, in terms of the
geometry of the dynamical system, for the crossed
product to be purely infinite (and hence a Kirchberg algebra, provided
that the conditions for nuclearity, separability and simplicity are
satisfied). Laca and Spielberg showed in \cite{LacaSpi:purelyinf} that
pure infiniteness of the crossed product is ensured by
requiring that the action is a \emph{strong boundary action}, meaning
that $X$ is infinite and that any two non-empty open subsets of $X$ can 
be translated by groups elements to cover the entire space $X$. 

Jolissaint and Robertson generalized this result and showed that it is 
enough to require the action is \emph{$n$-filling}, a similar 
property but with $n$ subsets instead of two
subsets of $X$. Their results also provide sufficient conditions
for pure infiniteness of crossed products $A \rtimes_r G$ when $A$ is
non-abelian, cf.\ \cite[Theorem 1.2]{Guyan:pi}. 
For other results on when crossed product 
\Cs s are Kirchberg  algebras we refer to \cite{Del, Kis:simple,
  LacaSpi:purelyinf}. 
 
In parts guided by the motivating example of the Roe algebra we shall
in this paper be interested in conditions on a dynamical system
$(A,G)$ that will ensure that the corresponding reduced crossed
product \Cs{} $A \rtimes_r G$ is purely infinite, but not necessarily
simple. We refer to \cite{KirRor} for more about non-simple purely
infinite \Cs s.

We show that the crossed product $A\rtimes_r G$ is purely infinite if
and only if every non-zero positive 
element in $A$ is properly infinite in 
$A\rtimes_r G$ provided that $G$ is discrete and exact, $A$
is separable and has the so-called ideal property (IP) (projections
separate ideals), and the action of $G$ on 
$\widehat{A}$ is essentially free (Theorem \ref{main.theorem}). In
the case where $A$ is the continuous functions on the Cantor set, we
show that $A\rtimes_r G$ is purely infinite if
and only if every non-zero projection in $A$ is properly infinite in 
$A\rtimes_r G$, here assuming that $G$ is discrete and exact and that 
the action of $G$ on
$\widehat{A}$ is essentially free, cf.\ Theorem
\ref{main.theorem.abelian}. 

In Section~\ref{sec:type} we consider dynamical systems $(C(X),G)$, where $X$
is totally disconnected, and we show that if a clopen subset $E$ of
$X$ is $G$-paradoxical (in a way that respects clopen sets), then $1_E$ is
a properly infinite projection in the crossed product $C(X) \rtimes_r
G$. Partial converse results are also obtained. In the case of the Roe
algebra $\ell^\infty(G) \rtimes_r G$ the situation is more clear: If $E$
is a subset of $G$, then $1_E$ is properly infinite in $\ell^\infty(G)
\rtimes_r G$ if and only if $E$ is $G$-paradoxical, which again happens if
and only if there is no (unbounded) trace on $\ell^\infty(G)
\rtimes_r G$ which is non-zero and finite on $1_E$.
In Section~\ref{sec:Roe} we
use this result to show that each countable, discrete, non-amenable,
exact group admits an action on the Cantor set such that the crossed
product is a Kirchberg algebra in the UCT class.

We thank Claire Anantharaman-Delaroche, 
George Elliott, Uffe Haagerup, Nigel Higson, Eberhard
Kirchberg, and Guoliang Yu for valuable discussions on the topics of
this paper.

\section{Notations and a preliminary result} \label{sec:notations}

Given a \Cs{} dynamical system $(A,G)$ with $G$ a discrete
group, let 
$A \rtimes_{r} G$ and $A \rtimes G $ denote the reduced
and the full crossed product \Cs s, respectively. In general we have a
surjective, and not necessarily injective, $^*$-homomorphism $A
\rtimes G \to A \rtimes_{r} G$. If $G$ is amenable,
or if $G$ acts amenably on $A$, then this $^*$-homomorphism is
injective, whence the reduced and the full crossed products are equal.
In this case the dynamical system is said to be regular. 

Let $C_c(G,A)$
be the common subalgebra of both crossed products consisting of finite
sums $\sum_{t\in G} a_tu_t$, where $a_t \in A$ (only finitely many
non-zero), and $t \mapsto u_t$, $t \in G$, is the unitary
representation of $G$ that implements the action of $G$ on
$A$. (If $A$ is unital, then each $u_t$ belongs to the crossed product,
and in general $u_t$  belongs to the multiplier algebra of the crossed
product.) We
suppress the canonical inclusion map $A\to A\rtimes_r G $ and
view $A$ as being a sub-\Cs{} of $A \rtimes_r G$. 

We have a (faithful) 
conditional expectation $E \colon A \rtimes_r G \to A$ which
on $C_c(G,A)$ is given by $\sum_{t\in G} a_t u_t \mapsto a_e$, where
$e \in G$ denotes the neutral element. 

For every $G$-invariant ideal $I$ in $A$,
the natural maps in the short exact sequence
\begin{eqnarray} \label{diagram0}
\xymatrix{
0 \ar[r] & I \ar[r]^-{\iota} & A \ar[r]^{\rho} & A/I \ar[r] & 0
}
\end{eqnarray}
extend canonically to maps at the level of reduced crossed
products, giving rise to the possibly non-exact sequence 
\begin{eqnarray}
\label{diagram}
	\xymatrix{0 \ar[r] & I\rtimes_r G  \ar[r]^-{\iota \rtimes_r
            \id } & A\rtimes_r G  \ar[r]^{\rho \rtimes_r \id } &
          A/I\rtimes_r G \ar[r] & 0 
}
\end{eqnarray}
cf.\ \cite[Remark 7.14]{Will:cross}. In general, the kernel of $\rho
\rtimes_r \id$, contains---but need not be contained in---the image of
$\iota \rtimes_r \id$. The sequence \eqref{diagram} is exact precisely
when $\ker(\rho\rtimes_r \id ) \subseteq {\mathrm{im}}(I\rtimes_r
G)$. By the definition of Kirchberg and Wassermann, \cite{KirWas},  
$G$ is exact if and only if \eqref{diagram} is exact for all exact
sequences \eqref{diagram0} of \Cs s with compatible $G$-actions.

The action of
$G$ on $A$ is said to be \emph{exact} if every $G$-invariant ideal in $A$
induces a short exact sequence at the level of reduced crossed
products, i.e.\ if \eqref{diagram} is exact for all $G$-invariant
closed two-sided ideals in $A$ (i.e., whenever \eqref{diagram0}
is exact with $A$ fixed), cf.\ \cite[Definition 1.2]{Sier}. Recall that $A$
is said to \emph{separate the ideals in $A\rtimes_r G$} 
if the (surjective) map
$J\mapsto J\cap A$, from the ideals in $A\rtimes_r G $ into the
$G$-invariant ideals in $A$, is injective. If the \Cs{} $A$
separates the ideals in 
$A\rtimes_r G$, then the action of $G$ on $A$ must be exact,
cf.\ \cite[Theorem 1.10]{Sier}. 

Recall that the action of $G$ on $\widehat{A}$ is said to be
\emph{essentially free} 
provided that for every closed $G$-invariant subset $Y\subseteq
\widehat{A}$, the subset of points in $Y$ with trivial isotropy is
dense in $Y$, cf.\ \cite{Ren:fixed}. It was shown in 
\cite{Sier} that $A$ separates the ideals in $A\rtimes_r G$ if $G$ is
a discrete group, the action of $G$ on $A$ is exact, and if the action
of $G$ on $\widehat{A}$ is essentially free. In particular, if a
discrete exact group $G$ acts (essentially) 
freely on a compact Hausdorff space $X$,
then $C(X)$ separates ideals in $C(X) \rtimes_r G$. 

We remind the reader of some basic definitions involving Cuntz
comparison and infiniteness of positive elements in a \Cs.
Let $a,b$ be positive elements in a \Cs{}
$A$. Write $a\precsim b$ if there exists a sequence $(r_n)$ in $A$
such that $r_n^*br_n\to a$. More generally for $a\in M_n(A)^+$
and $b\in M_m(A)^+$ write $a\precsim b$ if there exists a sequence
$(r_n)$ in $M_{m,n}(A)$ with $r_n^*br_n\to a$. For
$a\in M_n(A)$ and $b\in M_m(A)$ let $a\oplus b$ denote the element
$\textrm{diag}(a,b)\in M_{n+m}(A)$.  

A positive element $a$ in a \Cs{} $A$ is said to be \emph{infinite} if
there exists a non-zero positive element $b$ in $A$ such that $a\oplus
b\precsim a$. If $a$ is non-zero and if $a\oplus a\precsim a$, then
$a$ is said to be \emph{properly infinite}. This extends the usual
concepts of infinite and properly infinite projections, cf.\ \cite[p.\
642--643]{KirRor}.  

A \Cs{} $A$ is \emph{purely infinite} if there are no
characters on $A$ and if for every pair of positive elements $a,b$ in
$A$ such that $b$ belongs to the ideal in $A$ generated by $a$, 
one has $b\precsim a$.
Equivalently, a \Cs{} $A$ is purely infinite if every non-zero
positive element $a$ in $A$ is properly infinite, cf.\ \cite[Theorem
4.16]{KirRor}. 

We state below a general, but also rather restrictive, condition that
implies pure infiniteness of a crossed product. 

\begin{proposition}
\label{prop1}
Let $(A,G)$ be a C*-dynamical system with $G$ discrete. Suppose that
$A$ separates the ideals in $A\rtimes_r G$. Then $A\rtimes_r G$ is
purely infinite if and only if all non-zero positive elements in $A$ are
properly infinite in $A\rtimes_r G$ and
$E(a)\precsim a$ for all positive elements $a$ in $A\rtimes_r G$. 
\end{proposition}

\begin{proof}
Suppose that $A\rtimes_r G$ is purely infinite.  
Fix a non-zero positive element $a$ in $A\rtimes_r G$. By
\cite[Theorem 1.10]{Sier} we have that
$E(a)$ belongs to the ideal in $A \rtimes_r G$ generated by $a$,
whence $E(a)\precsim a$ by proper infiniteness of $a$, cf.\ 
\cite[Theorem 4.16]{KirRor}.

Now suppose that every non-zero positive elements in $A$ is properly infinite
in $A\rtimes_r G$ and that $E(a)\precsim a$ for every positive element
$a$ in $A\rtimes_r G$. Since the action of $G$ on $A$ is exact it
follows that $a$ belongs to the ideal in $A \rtimes_r G$ generated by
$E(a)$ for all positive elements $a$ in $A\rtimes_r G$, cf.\
\cite[Proposition 1.3]{Sier}. This implies that $a\precsim
E(a)$ using that $E(a)$ is properly
infinite, cf.\ \cite[Proposition 3.5 (ii)]{KirRor}. 
By the relation
$$a\oplus a\precsim E(a)\oplus E(a)\precsim E(a)\precsim a$$
in $A\rtimes_r G$, we conclude that $a$ is properly infinite. This shows that 
$A\rtimes_r G$ is purely infinite, cf.\ \cite[Theorem 4.16]{KirRor}. 
\end{proof}

\begin{remark}
The condition in Proposition~\ref{prop1} that $E(a) \precsim a$ for all
positive elements $a$ in $A \rtimes_r G$ is automatically 
satisfied in many cases of
interest, for example in Theorem~\ref{main.theorem}.
\end{remark}

\begin{remark}
In the proof of Proposition \ref{prop1} we used exactness and proper infiniteness of positive elements in $A$ to conclude that  $a \precsim E(a)$ for all positive elements $a$ in $A \rtimes_r G$. It is shown in \cite[Proposition 1.3]{Sier} that $a$ belongs to the closed two-sided ideal in $A \rtimes_r G$ generated by $E(a)$ for all positive $a$ in $A \rtimes_r G$ if and only if the action of $G$ on $A$ is exact. It follows in particular, that $a \precsim E(a)$ does not hold for all $a$ when the action of $G$ on $A$ is not exact. 

Each (positive) element $a$ in $A \rtimes_r G$ can be written as a formal sum $\sum_{t \in G} a_t u_t$, where $t \mapsto u_t$ is the unitary representation of $G$ in (the multiplier algebra of) $A \rtimes_r G$, and where $a_t = E(au_t^*) \in A$. Upon writing $a = xx^*$ and using Lemma \ref{lm:norm-convergence} below, one can show that each term, $a_tu_t$, in this formal sum belongs to the hereditary sub-\Cs{} of $A \rtimes_r G$ generated by $E(a)$. It therefore follows that $a$ belongs to the hereditary sub-\Cs{} of $A \rtimes_r G$ generated by $E(a)$, and hence that $a \precsim E(a)$, whenever $a$ belongs to $C_c(G,A)$, or, slightly more generally, when the formal sum, $a= \sum_{t \in G} a_t u_t$, is norm-convergent.

It remains a curious open problem if the relation $a \precsim E(a)$ holds for all positive $a$ in $A \rtimes_r G$, when the action of $G$ on $A$ is exact.
\end{remark}

\section{Pure infiniteness of crossed products with the ideal property} 

We provide a necessary condition for a crossed product $A
\rtimes_r G$ to be purely infinite assuming that $A$ has the so-called
ideal property (projections in $A$ separate ideals in $A$), and under the
assumption that the action of $G$ on $A$ is exact and essentially
free. (The latter notion is defined in Section~\ref{sec:notations}.)

\begin{lemma}
\label{lemma.A.IP.cross.IP}
Let $(A,G)$ be a $C^*$-dynamical system with $G$ discrete. Assume that
the action of $G$ on $A$ is exact and that the action of $G$ on
$\widehat{A}$ is essentially free. If $A$ has the ideal
property, then so does $A\rtimes_r G$.   
\end{lemma}

\begin{proof}
The assumptions on the action imply that $A$ separates ideals in ${A
\rtimes_r G}$, cf.\ \cite{Sier}. As projections in $A$
separates ideals in $A$, it easily follows that projections in $A$
(and hence also projections in $A \rtimes_r G$) separate ideals in $A
\rtimes_r G$. 
\end{proof}

\begin{lemma}
\label{projections.using.essentially.free}
Let $(A,G)$ be a  $C^*$-dynamical system with $A$ separable and 
$G$ countable and
discrete. Assume that the action of $G$ on $\widehat A$ is essentially free.
Then for every $G$-invariant ideal $I$ in $A$ and every non-zero
positive element $b$ in $A/I\rtimes_r G$ there exists a non-zero
positive element $a$ in $A/I$ such that $a\precsim b$.
\end{lemma}

\begin{proof} 
Let $I$ and $b$ be as above. We can without loss of generality
assume that $\|E(b)\| = 1$, where $E  \colon A/I\rtimes_r G \to A/I$
is the canonical (faithful) conditional expectation. 
Essential freeness of the action on
$\widehat{A}$ implies that the induced action of $G$ on $A/I$ is
topologically free and hence properly outer, see \cite{ArcSpi}. The
existence of an element  $x\in (A/I)^+$
satisfying 
\begin{eqnarray*}
\|x\|= 1, \quad  \|xE(b)x-xbx\|\leq 1/4, \quad \|xE(b)x\| \geq
\|E(b)\|-1/4 = 3/4.
\end{eqnarray*}
is ensured by \cite[Lemma 7.1]{OlePed3}.
With $a = (xE(b)x-1/2)_+$ we claim that $0\neq a\precsim xbx \precsim
b$. Indeed, the 
element $a$ is non-zero because $\|xE(b)x\|>1/2$, and $a\precsim
xbx$ holds since $\|xE(b)x-xbx\| < 1/2$, cf.\  \cite[Proposition
2.2]{Ror:uhfII}. 
\end{proof}

The conclusion of Lemma~\ref{projections.using.essentially.free} also
holds if $A$ is abelian and not necessarily separable (a case that
shall have our interest). The existence of the element $x$ in the
proof above was ensured by \cite[Lemma 7.1]{OlePed3} and the assumption
that $A/I$ is separable. In the case where $A$ is abelian and not
necessarily separable the existence of $x$ with the desired properties
was established in \cite[Proposition 2.4]{Exel}.

\begin{theorem}
\label{main.theorem}
Let $(A,G)$ be a C*-dynamical system with $G$ discrete. Suppose
that the action of $G$ on $A$ is exact and that the action of $G$ on
$\widehat{A}$ is essentially free. Suppose also that $A$ is 
separable and has the ideal property. Then the following statements are
equivalent:
\begin{enumerate} 
\item  Every non-zero positive element in $A$ is properly infinite in
$A\rtimes_r G$. 
\item The \Cs{} $A\rtimes_r G$ is purely infinite.
\item Every non-zero hereditary sub-\Cs{} in any quotient of
  $A\rtimes_r G$ contains an infinite projection. 
\end{enumerate}
\end{theorem}

\begin{proof}	
(ii) $\Leftrightarrow$ (iii). By Lemma \ref{lemma.A.IP.cross.IP} the
reduced crossed product $A\rtimes_r G$ has the ideal property. The
equivalence is obtained from \cite[Proposition 2.1]{PasRor}. 

(ii) $\Rightarrow$ (i). This is contained in Proposition~\ref{prop1}
(or in \cite[Theorem 4.16]{KirRor}). 	
	
(i) $\Rightarrow$ (iii). Fix an ideal $J$ in ${A\rtimes_r G}$ and a
non-zero hereditary sub-\Cs{} $B$ in 
the quotient $(A\rtimes_r G)/J$. We will show that $B$ contains an
infinite projection.  

By the assumptions on the action of $G$ on $A$ and on $\widehat{A}$
it follows from
\cite[Theorem 1.16]{Sier} that $A$ separates ideals in $A \rtimes_r
G$. Hence, 
$$(A\rtimes_r G)/J=(A/I)\rtimes_r G,$$
with $I=J\cap A$. Fix a non-zero
positive element $b$ in $B$. By
Lemma~\ref{projections.using.essentially.free} there exists a non-zero 
positive element $a$ in $A/I$ such that $a\precsim b$ relative to
$A/I\rtimes_r G$. By the assumption that $A$ has the ideal property we can
find a projection $q\in A$ that belongs to the ideal in $A$ generated
by the preimage of $a$ in $A$ but not to $I$. Then $q+I$ belongs to
the ideal in $A/I$ generated by $a$, whence $q+I \precsim a \precsim
b$ in $A/I \rtimes_r G$ 
(because $a$ is properly infinite by the assumption in (i)), see
\cite[Proposition 3.5 (ii)]{KirRor}. 

From
the comment after \cite[Proposition 2.6]{KirRor} we can find $z\in
A/I\rtimes_r G$ such that $q+I=z^*bz$. With $v=b^{1/2}z$ 
it follows that $v^*v= q+I$, whence $p := vv^* =
b^{1/2}zz^*b^{1/2}$ is a projection in $B$, which is equivalent to
$q$. By the assumption in (i), $q$ is properly infinite, and hence so
is $p$. 
\end{proof}

\section{Group actions on the Cantor set and paradoxical sets}

Recall that a commutative \Cs{} $C_0(X)$, with $X$ a locally compact
Hausdorff space, is of real rank zero if and only if $X$ is totally
disconnected (which again happens if and only if $C_0(X)$ has the
ideal property). We shall in this section study group actions on totally
disconnected spaces. First we give a sharpening of
Theorem~\ref{main.theorem} in the case where $A$ is abelian, of
real rank zero, but not necessarily separable: 

\begin{theorem}
\label{main.theorem.abelian}
Let $(A,G)$ be a  $C^*$-dynamical system with $A=C_0(X)$ and with $G$
discrete and exact. Suppose that the action of $G$ on $X$ is
essentially free and that $X$ is totally disconnected. Then the
following statements are equivalent: 
\begin{enumerate} 
\item Every non-zero projection in $A/I$ is infinite in $A/I\rtimes_r G$
for every $G$-invariant ideal $I$ in $A$.
\item Every non-zero projection in $A$ is properly infinite in
  $A\rtimes_r G$.  
\item The \Cs{} $A\rtimes_r G$ is purely infinite.
\end{enumerate}
\end{theorem}

\begin{proof}
(iii) $\Rightarrow$ (ii). Every non-zero projection in any purely
infinite \Cs{} is properly infinite. 

(ii) $\Rightarrow$ (i). Use that $A$ has real rank zero to lift a projection
in $A/I$ to a projection in $A$, cf.\ \cite[Theorem
3.14]{BroPed:realrank}. 

(i) $\Rightarrow$ (iii). Fix a non-zero hereditary sub-\Cs{} $B$
in the quotient of ${A\rtimes_r G}$ by some ideal $J$ in $A\rtimes_r
G$. By \cite[Proposition 4.7]{KirRor} we need only show that $B$
contains an infinite projection. 

By essential freeness of the action of $G$ on $X$ and exactness of $G$ we
have the identification  
$$(A\rtimes_r G)/J=(A/I)\rtimes_r G,$$
for $I:=J\cap A$, cf.\ \cite[Theorem 1.16]{Sier}. Fix a non-zero
positive element $b$ in $B$. By 
Lemma~\ref{projections.using.essentially.free} (and the remark below
that lemma) there exists a
non-zero positive element $a$ in $A/I$ such that $a\precsim b$
relative to $A/I\rtimes_r G$. The hereditary sub-\Cs{} $\overline{a
  (A/I)a}$ contains a non-zero projection $q$ by the assumption on
$A$, and $q$ is infinite in $A/I \rtimes_r G$ by the assumption in
(i). As $q \precsim b$ (and because $q$ is a projection) we can find
$z\in A/I\rtimes_r G$ such that $q=z^*bz$. It follows that $p:=
b^{1/2}zz^*b^{1/2}$ is a projection in $B$, which is equivalent to
$q$, and hence is infinite. 
\end{proof}

In lieu of the previous
result we are interested in knowing for
which clopen subsets $V$ of $X$ the projection $1_V$ in $C(X)$ is
properly infinite in $C(X) \rtimes_r G$. A sufficient, and sometimes
also necessary, condition for this to happen is that $V$ is
$G$-paradoxical in $X$ with respect to the open subsets of $X$. We
give a formal definition of this phenomenon below:

\begin{definition}
\label{def.paradox}
Given a discrete group $G$ acting on a topological space $X$ and a family
$\E$ of subsets of $X$, a non-empty set $U\subseteq X$ is called
\emph{$(G,\E)$-paradoxical} if there exist non-empty sets
$V_1,V_2,\dots$, $V_{n+m}\in \E$ and elements
$t_1,t_2,\dots,t_{n+m}$ in $G$ such that $\bigcup_{i=1}^nV_i  = 
\bigcup_{i=n+1}^{n+m}V_i  = U$ and such that $\big(t_k.V_k\big)_{k=1}^{n+m}$
are pairwise disjoint subsets of $U$.
\end{definition}

A projection $p$ in a \Cs{} $A$ is properly infinite if and only if
there exist partial isometries $x$ and $y$ in $A$ such that $x^*x=y^*y = p$
and $xx^* + yy^* \le p$. In the following let $\tau_X$ denote
the family of open subsets of the topological space $X$.

\begin{proposition}
\label{lemma.paradoxical.properly.infinite}
Let $X$ be a locally compact Hausdorff space, and let $G$ be a discrete group
acting on $X$. Suppose that $U$ is a compact open subset of $X$. Then
$U$ is $(G,\tau_X)$-paradoxical if and only if there exist elements
$x,y$ in $C_c(G,C_0(X)^+)$ such that $x^*x=y^*y = 1_U$ and $xx^*+yy^*
\le 1_U$. In this case $1_U$ is properly infinite in $C_0(X)
\rtimes_r G$.
\end{proposition}

\begin{proof}
Suppose first that $U$ is  $(G,\tau_X)$-paradoxical, and let 
$(V_i, t_i)_{i=1}^{n+m}$ be a system of subsets of
$U$ and elements in $G$ that witnesses the paradoxicality.  
Find partitions of unity $(h_i)_{i=1}^{n}$ and $(h_i)_{i=n+1}^{n+m}$ 
for $U$ relative to the open covers $(V_i)_{i=1}^n$ and
$(V_i)_{i=n+1}^{n+m}$, respectively. As $U$ is compact and open, we
can assume that each $h_i$ has support contained in $U$ so that
$\sum_{i=1}^n h_i = \sum_{i=n+1}^{n+m} h_i = 1_U$. Set
$$x=\sum_{i=1}^n u_{t_i}h_i^{1/2}, \qquad y=\sum_{i=n+1}^{n+m}
u_{t_i}h_i^{1/2}.$$ 
Notice that
$$h_{i}^{1/2} u^*_{t_i}u_{t_j}h_j^{1/2} = u^*_{t_i}
(t_i.h_i^{1/2})(t_j.h_j^{1/2})u_{t_j}=0$$
when $i \ne j$, since $\supp(t_k.h^{1/2}_k) \subseteq t_k.V_k$. It
follows that
$$x^*x=\sum_{i,j=1}^n h_{i}^{1/2}u_{t_i}^*u_{t_j}h_{j}^{1/2} \, = \, 
\sum_{i=1}^n h_{i}^{1/2}u_{t_i}^*u_{t_i}h_{i}^{1/2}
\, = \, \sum_{i=1}^n h_{i} =1_U.$$ 
In a similar way we see that $y^*y= 1_U$, $y^*x = 0$, and hence
$xx^*+yy^* \le 1_U$.

Assume, conversely, that we are given two elements 
$$x=\sum_{t\in F} u_{t}h_t^{1/2}, \qquad y=\sum_{s\in F'}
u_{s}g_s^{1/2}$$ in $C_c(G,C_0(X)^+)$
satisfying
$x^*x=y^*y=1_U$ and $xx^*+yy^*\le 1_U$, where
$F,F'\subseteq G$ are finite and $h_t $, $g_s$ belong to $C_0(X)^+$. Put
$$V_t =\{\xi \in X\colon h_t(\xi)\neq 0\}, \qquad W_s =\{\xi\in X\colon
g_s(\xi)\neq 0\}.$$  
We show that the system $\{(V_t,t) : t \in F\} \cup \{(W_s,s) : s \in
F'\}$ witnesses the $(G,\tau_X)$-paradoxicality of $U$. To this end,
first note that 
$$1_U=x^*x=E(x^*x)=\sum_{t,s\in F}
E(h_t^{1/2}u_{t}^*u_{s}h_s^{1/2})=\sum_{t\in F}
E(h_t^{1/2}u_{t}^*u_{t}h_t^{1/2})=\sum_{t\in F}h_t.$$ 
This implies that $\cup_{t \in F} V_t = U$. In a similar way we see
that $\cup_{s \in F'} W_s = U$.

For $r\neq e$ in $G$ we have that
$$0 = E(1_Uu_r)=E(x^*xu_r)=\sum_{t,s\in
  F}E(h_t^{1/2}u_{t}^*u_{s}h_s^{1/2}u_r).$$ 
Each term in the sum of the right-hand side of the equation above
is a positive element in $C_0(X)$, and must hence be zero. We can therefore conclude that 
$h_t^{1/2}u_{t}^*u_{s}h_s^{1/2}u_{t^{-1}s}^* = 0$
for all $s,t \in F$ with $t\neq s$. 
This shows that $t.h_t \perp s.h_s$ for $t\neq s\in F$, which again
implies that $t.V_t$ and $s.V_s$ are disjoint, when $s,t \in F$ and $s
\neq t$. In a similar
way we obtain that $t.W_t \cap s.W_s = \emptyset$ 
for $t\neq s\in F'$ and that $t.V_t \cap s.W_s  = \emptyset$ for $t\in F$ and $s\in
F'$ (the last property is obtained 
from the fact that $0=E(x^*yu_r)$ for every $r\in G$). 

Finally we show that $t.V_t\subseteq U$ for every $t\in F$. Since $1_U
E(xx^*) = E(1_Uxx^*)$ and $E(xx^*) = \sum_{t \in F} t. h_t$,
we can conclude that $t.h_t = 1_U (t.h_t)$, which yields the desired
inclusion. In a similar way we see that $s.W_s \subseteq U$ for all $s
\in F'$.
\end{proof}
\noindent
Theorem~\ref{main.theorem.abelian} and
Proposition~\ref{lemma.paradoxical.properly.infinite} imply the
following: 

\begin{corollary}
\label{ess.free.pi}
Let $(A,G)$ be a C*-dynamical system with $A=C(X)$ and $G$ discrete
and exact. Suppose that the action of $G$ on $X$ is essentially free
and $X$ has a basis of clopen $(G,\tau_X)$-paradoxical sets.
Then $C(X)\rtimes_r G$ is purely infinite.
\end{corollary}

\begin{remark}
If an action of a discrete group $G$ on a 
compact Hausdorff space $X$ is $n$-filling in the sense of Jolissaint
and Robertson, \cite{Guyan:pi}, (see also the introduction), 
then every non-empty 
open subset of $X$ is $(G,\tau_X)$-paradoxical and the action 
is minimal. Minimality is not required in Corollary 
\ref{ess.free.pi} thus giving a new characterization of  pure
infiniteness for this class of non-simple crossed products \Cs{}s.
\end{remark}

\begin{remark}
Let $G$ be a countable group acting on the Cantor set $X$.
 Assume that the full crossed product $C(X)\rtimes
G$ is simple and that all non-zero projections in $C(X)$ are properly
infinite in 
$C(X)\rtimes G$. Then $C(X)\rtimes G$ is purely infinite. Indeed,
simplicity of $C(X)\rtimes G$ trivially implies that the action of $G$
on $C(X)$ is exact and regular (the canonical surjection  $C(X)\rtimes G \to
  C(X)\rtimes_r G$ is injective). We can therefore infer pure infiniteness
of $C(X) \rtimes G$ from  \cite[Remark 4.3.7]{Sier:phd}. 
\end{remark}

\section{The type semigroup $S(X,G,\E)$} 
\label{sec:type}

We shall here study pure infiniteness of the crossed product \Cs{} $C(X)
\rtimes_r G$ coming from an action of a discrete group $G$ on a
compact Hausdorff space $X$ (typically the Cantor set) by studying the
associated \emph{type semigroup} considered in \cite{Wag:B-T}. 

A state on a preordered monoid $(S,+,\leq)$ is a map $f\colon S\to
[0,\infty]$ which respects $+$ and $\leq$ and fulfills that
$f(0)=0$. A state is said to be \emph{non-trivial} if it takes a value
different from $0$ and $\infty$. An element $x \in S$ is said to be \emph{properly infinite} if $2x \le x$; and the monoid $S$ is said to be \emph{purely
  infinite} if every $x\in S$ is properly infinite. The monoid $S$ is
\emph{almost unperforated} if, whenever $x,y\in S$ and $n,m\in \N$ are
such that $nx\leq my$ and $n>m$, then $x\leq y$. 

Let $\E$ be a $G$-invariant subalgebra of $\cP(X)$, the power set of
$X$. We write $S(X,G,\E)$ for
the \emph{type semigroup} of the induced action of $G$ on $\E$, where only
sets from $\E$ can by used to witness the equidecomposability of sets
in $\E$, cf.\ \cite[p.116]{Wag:B-T}. In more detail, the set
$S(X,G,\E)$ is defined to be
$$\Big\{\bigcup_{i=1}^n A_i\times \{i\}\colon A_i\in \E, \, n\in
\N \Big\} \big/\sim_S,$$ 
where the equivalence relation $\sim_S$ is defined as follows: Two
sets $A =\bigcup_{i=1}^n A_i\times \{i\}$ and
$B=\bigcup_{j=1}^m B_j\times \{j\}$ are equivalent, denoted
$A\sim_S B$, if there exist $l\in \N$, $C_k \in \E$, $t_k \in G$,
and natural numbers $n_k, m_k$ (for $k=1,2,\dots, l$) such that
$$A=\bigsqcup_{k=1}^l C_k\times \{n_k\}, \qquad  B=\bigsqcup_{k=1}^l
t_k.C_k\times \{m_k\},$$
(we use the symbol $\sqcup$ to emphasize the union being disjoint).
The equivalence class containing $A$ is denoted by $[A]$, and addition
is defined by
$$\Big[\bigcup_{i=1}^n A_i\times \{i\}\Big]+\Big[\bigcup_{j=1}^m B_j\times
\{j\}\Big] = 
\Big[\bigcup_{i=1}^{n} A_i \times \{i\} \, \cup \, 
\bigcup_{j=1}^m B_j\times \{n+j\} \Big].$$ 
For $E \in \E$, we shall often write $[E]$ instead of $[E \times
\{1\}]$. The semigroup $S(X,G,\E)$ has neutral element $0 =[\emptyset]$, and it
is equipped with the \emph{algebraic preorder} ($x\leq y$ if $y=x+z$
for some $z$). This makes $S(X,G,\E)$ into a preordered monoid. 

\begin{lemma}
\label{state.to.measure}
Let $G$ be a discrete group acting on a compact Hausdorff space $X$. Let $\E$
denote the family of clopen subsets of $X$. Suppose that
$\sigma(\E)=\B(X)$ (the Borel $\sigma$-algebra on $X$). 
Then  every non-trivial state $f$ on $S(X,G,\E)$ lifts
to a $G$-invariant measure $\mu$ on $(X,\B(X))$ such that $0<\mu(F)<\infty$
for some $F\in \E$. 
\end{lemma}

\begin{proof}
Define $\mu_0 \colon \E\to [0,\infty]$ by $\mu_0(F) =f([F])$ for $F\in \E$.
We show that $\mu_0$ is a \emph{premeasure} in the sense of \cite[p.\
30]{Foll}. We trivially have that $\mu_0(\emptyset)=f(0)=0$. Suppose that
$(F_i)_{i=1}^\infty$ is a sequence  of pairwise disjoint sets in $\E$ such that
$F =\bigcup_{i=1}^\infty F_i$ belongs to $\E$. Then $F = \bigcup_{i=1}^n
F_i$ for some $n$ by compactness of $F$ (so $F_i = \emptyset$ for all $i >
n$). Additivity of $f$ now implies that $\mu_0(F) = \sum_{i=1}^\infty
\mu_0(F_i)$. This shows that $\mu_0$ is a premeasure. 
By \cite[Theorem 1.14]{Foll}, $\mu_0$ extends (uniquely) to a measure $\mu$ on
$\B(X)$. As $\mu(F) = \mu_0(F) = f([F])$ for all $F \in \E$, the
existence of a clopen subset $F$ of $X$ such that  
$0<\mu(F)<\infty$ follows from the fact that $f$ is non-trivial. 

Let $t \in G$ be given, and let $t.\mu$ be the Borel measure given by
$(t.\mu)(E) = \mu(t^{-1}.E)$. Then, for every $E \in \E$,
$$(t.\mu)(E) = \mu(t^{-1}.E) = \mu_0(t^{-1}.E) = f([t^{-1}.E]) = f([E]) = \mu_0(E).$$ 
By uniqueness of $\mu$ we must have $t.\mu = \mu$, so $\mu$ is
$G$-invariant.
\end{proof}

\noindent It is well-known that a bounded invariant trace on a \Cs, with 
a discrete group acting on it, extends to a (bounded) trace on the crossed product. 
We shall need to extend unbounded invariant tracial weights to the crossed product, cf.\ 
Lemma \ref{measure.to.trace} below. Whereas this result may be known to
experts, we failed to find a suitable reference for it, so we include the following two
lemmas. The second named author thanks George Elliott and Uffe
Haagerup for explaining the trick below (how to use Dini's theorem to
obtain norm-convergence from convergence on states) and its
application to Lemma~\ref{lm:norm-convergence}. 

Let $A$ be a \Cs, and denote by $S(A)$ its state space (i.e., the set of
all positive linear functionals on $A$ of norm 1). For each
self-adjoint element $a$ in $A$ let $\hat{a} \in C(S(A),\R)$ denote the
function given by $\hat{a}(\rho) = \rho(a)$, $\rho \in S(A)$. Note
that $\|a\| = \|\hat{a}\|_\infty$.  Suppose that $a$ and $(a_n)$
are positive elements in $A$ such that $\rho(a) = \sum_{n=1}^\infty
\rho(a_n)$ for all $\rho \in S(A)$. Then $\hat{a} = 
\sum_{n=1}^\infty \hat{a}_n$,
and the sum is point-wise convergent, and hence uniformly
convergent by
Dini's theorem. It follows that $a = \sum_n a_n$, and that the sum is 
norm-convergent. 

\begin{lemma} \label{lm:norm-convergence}
Let $A$ be a \Cs, and let $G$ be a countable discrete group acting on
$A$. For each element $x$ in the reduced crossed product $A \rtimes_r
G$ and for each $t \in G$ let $x_t = E(xu_t^*) \in A$. It follows that
$$E(xx^*) = \sum_{t \in G} x_tx_t^*, \qquad E(x^*x) = \sum_{t \in G}
t.({x_{t^{-1}}}^*x_{t^{-1}}),$$
and the sums are norm-convergent.
\end{lemma}

\begin{proof} Let $\pi \colon A \to B(H)$ be the universal
  representation of $A$, and let $$\pi \times \lambda \colon A \rtimes_r
  G \to B(\ell^2(G,H))$$ be the corresponding left regular
  representation of the reduced crossed product. For each $s \in G$,
  let $V_s \colon H \to 
  \ell^2(G,H)$ be the isometry given by 
$$V_s(\xi)(t) = \begin{cases} \xi, & t=s^{-1},\\ 0, & t \ne s^{-1}, 
\end{cases}
\qquad \xi \in H, \quad t \in G.$$
It is straightforward to check 
that $\sum_{t \in G} V_tV_t^* =I$ (the sum is strongly convergent), and that
$\pi(E(au_s^*)) = V_e^*(\pi \times \lambda)(a)V_s$
for all $a$ in $A \rtimes_r G$ and all $s \in G$. It follows that 
\begin{equation} \label{eq:n1}
 \pi(E(xx^*))  = V_e^* (\pi \times \lambda)(x) \Big({\textstyle{\sum_{t
      \in G}}} V_tV_t^*\Big) 
(\pi \times \lambda)(x^*) V_e = \sum_{t \in G} \pi(x_tx_t^*),
\end{equation}
where the sum is strongly convergent, and similary,
\begin{equation} \label{eq:n2}
\pi(E(x^*x)) = \sum_{t \in G} \pi(t.({x_{t^{-1}}}^*x_{t^{-1}})).
\end{equation}
For each $\rho \in S(A)$ there is a
vector state $\varphi$ on $B(H)$ such that $\rho = \varphi \circ \pi$
(because $\pi$ is the universal representation). As vector states are
strongly continuous, we can apply $\varphi$ to Equations \eqref{eq:n1}
and \eqref{eq:n2} above to obtain that
$$\rho(E(xx^*)) = \sum_{t \in G} \rho(x_tx_t^*), \qquad \rho(E(x^*x)) =
\sum_{t \in G} \rho(t.(x_{t^{-1}}^*x_{t^{-1}})),$$ 
for all $\rho \in S(A)$. 

The conclusion now follows from the comment above the lemma.
\end{proof}

\noindent
We shall refer to a map $\varphi$ from the positive cone, $A^+$,
of a \Cs{} $A$ to $[0,\infty]$ as being a tracial weight if it
satisfies $\varphi(\alpha a + \beta b) = \alpha \varphi(a) + \beta
\varphi(b)$ and $\varphi(x^*x)=\varphi(xx^*)$ for all $a,b \in A^+$,
all $\alpha, \beta \in \R^+$, and all $x \in A$. (A tracial weight
as above extends to an unbounded linear trace on the algebraic ideal of
$A$ generated by the set of positive elements $a \in A$ with
$\varphi(a) < \infty$.)

\begin{lemma}
\label{measure.to.trace}
Let $G$ be a countable 
discrete group acting on a \Cs{} $A$, and let $\varphi_0
\colon A^+ \to [0,\infty]$ be a $G$-invariant, lower
semi-continuous tracial weight. It follows that
the mapping $\varphi = \varphi_0 
\circ E \colon (A \rtimes_r G)^+ \to [0,\infty]$ is a lower
semi-continuous tracial weight which extends $\varphi_0$. 
\end{lemma}

\begin{proof} It is trivial to verify that $\varphi$ is $\R^+$-linear
  and lower semi-continuous (the latter because $E$ is 
  continuous). It is also clear 
  that $\varphi$ extends $\varphi_0$. We must show that 
$\varphi(x^*x) = \varphi(xx^*)$ for all
$x \in A \rtimes_r G$. By Lemma~\ref{lm:norm-convergence}, 
by the assumption that $\varphi_0$ is lower
semi-continuous (and additive), which entails that it respects 
norm-convergent sums of positive elements, 
and by $G$-invariance and the trace property of
$\varphi_0$, it follows that
\begin{eqnarray*}
\varphi(x^*x) & = & \varphi_0(E(x^*x)) \, = \, 
\varphi_0 \Big(\sum_{t \in G} t.({x_{t^{-1}}}^*x_{t^{-1}})\Big) \\ & = &  
\sum_{t \in G} \varphi_0(t.({x_{t^{-1}}}^*x_{t^{-1}}))
\, = \, \sum_{t \in G} \varphi_0(x_t x_t^*) \\ & = & 
\varphi_0\Big(\sum_{t \in G} x_tx_t^*\Big) \, = \, \varphi_0(E(xx^*))\, =
\, \varphi(xx^*),
\end{eqnarray*}
as desired.  
 \end{proof}

\begin{theorem}
\label{trace.pi}
Let $(C(X),G)$ be a C*-dynamical system with $G$ countable, 
discrete, exact, and
$X$ the Cantor set. Let $\E$ denote the family of clopen subsets of
$X$. Suppose the action of $G$ on $X$ is essentially free. Consider
the following properties:
\begin{enumerate}
\item The semigroup $S(X,G,\E)$ is purely infinite.
\item Every clopen subset of $X$ is $(G,\tau_X)$-paradoxical.
\item The \Cs{} $C(X)\rtimes_r G$ is purely infinite.
\item The \Cs{} $C(X)\rtimes_r G$ is traceless\footnote{A \Cs{} is in
    \cite{KirRorOinf} said to be traceless if it admits no non-zero 
    lower semi-continuous (possibly unbounded) 2-quasitraces defined
    on an algebraic ideal of the given \Cs.}.
\item There are no non-trivial states on $S(X,G,\E)$.
\end{enumerate}
Then \emph{(i)} $\Rightarrow$ \emph{(ii)} $\Rightarrow$ \emph{(iii)} 
$\Rightarrow$ \emph{(iv)} $\Rightarrow$ \emph{(v)}. Moreover, if
$S(X,G,\E)$ is almost unperforated, then 
\emph{(v)} $\Rightarrow$ \emph{(i)}, whence all five properties are equivalent.
\end{theorem}

\begin{proof}
(i) $\Rightarrow$ (ii). Let $E\in \E$ be given. We show that $E$ is
$(G,\tau_X)$-paradoxical if $2[E] \le [E]$. Suppose that $2[E] \le
[E]$, and find $F \in \E$ such that
$2[E]+[F]=[E]$. Then
$$E \times \{1\} \, \sim_S \, (E \times \{1,2\}) \, \cup \, (F \times \{3\}).$$
It follows that there exists $\ell \in \N$, $A_j \in \E$, $t_j \in G$,
and $m_j \in \{1,2,3\}$ for $j=1,2, \dots, \ell$ such that 
$$E \times \{1\} = \bigsqcup_{j=1}^\ell A_j \times \{1\}, \quad (E \times
\{1,2\}) \, \cup \, (F \times \{3\})
=  \bigsqcup_{j=1}^\ell t_j.A_j \times \{m_j\}.$$ This entails that
the sets $A_j$ are pairwise disjoint subsets of $E$. 
Renumbering the $A_j$'s 
so that $m_j = 1$ for $1 \le j \le n$, $m_j = 2$ for $n <
j \le n+m$, and $m_j = 3$ for $n+m < j \le \ell$, we obtain
$$\bigsqcup_{j=1}^n t_j.A_j = \bigsqcup_{j=n+1}^{n+m} t_j.A_j = E.$$
This shows that $E$ is $(G,\tau_X)$-paradoxical.

(ii) $\Rightarrow$ (iii). Corollary \ref{ess.free.pi}.

(iii) $\Rightarrow$ (iv). We refer to \cite{KirRorOinf}.

(iv) $\Rightarrow$ (v). Suppose $f$ is a non-trivial state on the
preordered monoid $S(X,G,\E)$. Since $X$ is totally disconnected and
second countable, we have
$\B(X)=\sigma(\E)$. By Lemma \ref{state.to.measure} there exists a
$G$-invariant Borel measure $\mu$ on $X$ and a clopen subset $F$ of
$X$ such that $0<\mu(F)<\infty$. By Lemma \ref{measure.to.trace},
the (tracial) lower semi-continuous weight
$$\varphi_0(f) = \int_X f \, d\mu, \qquad f \in C(X)^+,$$
extends to a lower semi-con\-tin\-uous tracial weight
$\varphi \colon (C(X)\rtimes_r G)^+ \to [0,\infty]$. 
Restrict (and extend) $\varphi$ to a (linear, unbounded, lower
semi-con\-tin\-uous) trace $\tau$ defined on the algebraic ideal
generated by all 
positive elements $a$ in $C(X) \rtimes_r G$ with $\varphi(a) <
\infty$. As $0 < \varphi(1_F) < \infty$, the element $1_F$ belongs to
the domain of $\tau$ and $\tau(1_F) = \varphi(1_F) > 0$, whence $\tau$
is non-zero. Hence $C(X) \rtimes_r G$ is not traceless.

(v) $\Rightarrow$ (i). This implication holds for any preordered
almost unperforated abelian semigroup $(S,+,\le)$. Take any (non-zero) $x$ in
$S$. By the assumption that there is no state on $f$ normalized at $x$
it follows that a multiple of $x$ is properly infinite. This is a
consequence of the Goodearl-Handelman theorem, \cite{GH}, and does not
use the assumption that $S$ is almost unperforated. (See for 
example \cite[Proposition 2.1]{OPR} where it is shown that absence
of states on $S$ normalized at $x$ implies that $2x <_s x$ and hence
that $2(k+1)x \le kx$ for some natural number $k$.) Now, if $kx$ is
properly infinite, then $\ell x \le kx$ for all natural numbers
$\ell$. This implies in particular that $2(k+1)x \le kx$, and so $2x
\le x$ since $S$ is almost unperforated. This shows that $S$ is
purely infinite. 
\end{proof}

We do not know to what extend $S(X,G,\E)$  is almost unperforated (or purely infinite), when 
$X$, $G$, and $\E$ are as in Theorem~\ref{trace.pi}. We know of no examples where
$S(X,G,\E)$  is not almost unperforated (in the case where $X$ is the Cantor set), but
we suspect that such examples may exist. In many cases of dynamical systems $(X,G)$, 
for which it is known that $C(X) \rtimes_r G$ is purely infinite (and simple), it is also known
(or easy to see) that the type semigroup $S(X,G,\E)$ is purely infinite in the case where $X$ is 
totally disconnected and $\E$ is the algebra of clopen subsets of $X$. This is for example the case 
for the strong boundary actions considered by Laca and Spielberg and the $n$-filling actions 
concidered by Jolissaint and Robertsen (mentioned in the introduction). 
In Section \ref{sec:Roe} we give examples 
of group actions on the Cantor set such that $C(X) \rtimes_r G$ is purely infinite, where 
$G$ can be any countable discrete non-amenable exact group. This does not immediately imply
that $S(X,G,\E)$ is purely infinite, but one can design the action of $G$ on $X$ such that 
 $S(X,G,\E)$ is purely infinite (by invoking a strengthened version of Proposition \ref{Roe-2}). 

We do not know if every
clopen subset $E$ of $X$ for which $1_E$ is properly infinite in $C(X)
\rtimes_r G$ is necessarily $G$-paradoxical (not to mention
$(G,\tau_X)$-paradoxical). However, we have the following result that
relies on Tarski's theorem. 
(The term "tracial weight" is defined above Lemma~\ref{measure.to.trace}.)

\begin{proposition} \label{Roe} Let $G$ be a countable 
discrete group that acts on a discrete set $X$, and let $E
  \subseteq X$ be given. Then the following conditions are equivalent:

\begin{enumerate}
\item $E$ is $G$-paradoxical.
\item For every tracial weight $\varphi \colon (\ell^\infty(X)
  \rtimes_r G)^+ \to [0,\infty]$, the value of $\varphi(1_E)$ is either
  $0$ or $\infty$.
\item $1_E$ is properly infinite in $\ell^\infty(X) \rtimes_r G$.
\item The $n$-fold direct sum $1_E \oplus 1_E \oplus \cdots \oplus
  1_E$ is properly infinite  in $M_n(\ell^\infty(X) \rtimes_r G)$ for some $n$.
\end{enumerate}
\end{proposition}

\begin{proof}
(ii) $\Rightarrow$ (i). 
We note first that if $\nu$ is a finitely additive measure on $\cP(X)$
with $0 < \nu(X) < \infty$, then there is a bounded positive linear
functional $\psi$ on $\ell^\infty(X)$ such that $\psi(1_E) =
\nu(E)$ for all $E \subseteq X$. Indeed, first define $\psi$ on simple
functions by
$$\psi\big( \sum_{j=1}^n \alpha_j 1_{E_j}\big) = \sum_{j=1}^n
\alpha_j \nu(E_j), \qquad \alpha_j \in \C, \, E_j \in \cP(X).$$
Noting that $\psi$ is positive, linear, and bounded (with bound
$\nu(X)$) on
the uniformly dense subspace of $\ell^\infty(X)$ consisting of simple
functions, we can by continuity extend $\psi$ to $\ell^\infty(X)$ 
to the desired functional. 

Assume that $E$ is not $G$-paradoxical. It then follows from Tarski's
theorem, cf.\ \cite[p.\ 116]{Pat:amen}, that there is a finitely 
additive $G$-invariant measure $\mu$ on $\cP(X)$ satisfying
$\mu(E)=1$. Let $\cP_{\mathrm{fin}}(X)$ be the (upwards directed) 
set of all subsets $F$ of
$X$ with $\mu(F) < \infty$. For each $F \in \cP_{\mathrm{fin}}(X)$ let
$\mu_F$ be the finitely additive bounded measure on $\cP(X)$ given by $\mu_F(A)
= \mu(A \cap F)$, and let $\psi_F$ be the associated bounded linear
functional on $\ell^\infty(X)$ constructed above. Notice that $\psi_F
\le \psi_{F'}$ if $F \subseteq F'$. Put
$$\varphi_0(f) = \sup_{F \in  \cP_{\mathrm{fin}}(X)} \psi_F(f), \qquad f
\in \ell^\infty(X)^+.$$

Then $\varphi_0$ is $\R^+$-linear (being the supremum of an upwards
directed family of $\R^+$-linear functionals) and 
lower semicontinuous (being a
supremum of a family of continuous functionals). We have $\psi_F(1_E) =
\mu_F(E) = \mu(E)$ for all $F \in \cP_{\mathrm{fin}}(X)$ with $E
\subseteq F$. This shows that $\varphi_0(1_E) = \mu(E) = 1$. To prove
$G$-invariance note first that 
$\psi_{F}(t.f) = \psi_{t^{-1}F}(f)$ for all $t \in G$ and all $f \in
\ell^\infty(X)$. Indeed, by linearity and continuity it suffices to
verify this for characteristic functions. Use here that $t.1_E =
1_{tE}$ and that $\mu_F(tE) = \mu(F \cap tE) = \mu(t^{-1}F \cap
E)=\mu_{t^{-1}F}(E)$. Thus 
$$\varphi_0(t.f) = \sup_{F \in  \cP_{\mathrm{fin}}(X)} \psi_F(t.f)
 = \sup_{F \in  \cP_{\mathrm{fin}}(X)} \psi_{t^{-1}F}(f)
 = \sup_{F \in  \cP_{\mathrm{fin}}(X)} \psi_{F}(f) = \varphi_0(f)$$
for all $f$ in $\ell^\infty(X)^+$, thus showing that $\varphi_0$ is
$G$-invariant. 

To complete the proof, 
use Lemma~\ref{measure.to.trace} to extend $\varphi_0$ to a  tracial
weight $\varphi \colon (\ell^\infty(X) \rtimes_r G)^+ \to
[0,\infty]$. As $\varphi(1_E) = \varphi_0(1_E) = 1$ we conclude that
(ii) does not hold. 

(i) $\Rightarrow$ (ii). Suppose that (ii) does not hold and that
$\varphi \colon (\ell^\infty(X) \rtimes_r G)^+ \to [0,\infty]$ is a tracial
weight with $\varphi(1_E)=1$. Define a finitely additive measure
$\mu$ on $\cP(X)$ by $\mu(F) = \varphi(1_F)$. Then $\mu(tF) = \varphi(t.1_F) =
\varphi(u_t1_Fu_t^*) = \varphi(1_F) = \mu(F)$ by the tracial property
of $\varphi$, which shows that $\mu$ is $G$-invariant. By (the easy part of)
Tarski's theorem, the existence of a $G$-invariant finitely additive
measure $\mu$ on $\cP(X)$ with $\mu(E)=1$ implies that $E$ is not
$G$-paradoxical.

(i) $\Rightarrow$ (iii) is contained in
Proposition~\ref{lemma.paradoxical.properly.infinite}.  

(iii) $\Rightarrow$ (iv) is trivial.

(iv) $\Rightarrow$ (ii). Assume that the $n$-fold direct sum 
$$P:=1_E \oplus 1_E \oplus \cdots \oplus 1_E 
\in M_n(\ell^\infty(X) \rtimes_r G)$$ 
is properly infinite. Suppose that
$\varphi \colon (\ell^\infty(X) \rtimes_r G)^+ \to [0,\infty]$ is a
tracial weight. Extend $\varphi$ to the positive cone over any
matrix algebra over $\ell^\infty(X) \rtimes_r G$ in the usual way by the
formula $\varphi\big((x_{ij})\big) = \sum_i \varphi(x_{ii})$. Then
$\varphi(P) = n\varphi(1_E)$. As $P$ is properly infinite there exist
projections $P_1,P_2,Q$ in $M_n(\ell^\infty(X) \rtimes_r G)$ such that
$P = P_1+P_2+Q$ and $P \sim P_1 \sim P_2$. Thus 
$$\varphi(P) = \varphi(P_1) + \varphi(P_2) + \varphi(Q) = 2\varphi(P)
+ \varphi(Q),$$
whence $\varphi(P)$ is either $0$ or $\infty$. It follows that we
cannot have $0 < \varphi(1_E) < \infty$, thus showing that (ii)
holds. 
\end{proof}

\begin{corollary} \label{cor:Roe} Let $G$ be a countable discrete non-amenable group and let $p$ be a projection in $\ell^\infty(G)$ that is full in $\ell^\infty(G) \rtimes_r G$ (i.e., is not contained in a proper ideal in $\ell^\infty(G) \rtimes_r G$). Then $p$ is properly infinite.
\end{corollary}

\begin{proof}  The assumption that $p$ is full in  $\ell^\infty(G) \rtimes_r G$ implies that the $n$-fold direct sum $p \oplus p \oplus \cdots \oplus p$ dominates the unit $1$ of $\ell^\infty(G) \rtimes_r G$ for some natural number $n$. As $G$ is non-amenable, $G$ is $G$-paradoxical, whence $1 = 1_G$ is properly infinite in $\ell^\infty(G) \rtimes_r G$, cf.\ Proposition \ref{Roe}. It follows that
$$1 \precsim p \oplus p \oplus \cdots \oplus p \precsim 1 \oplus 1 \oplus \cdots \oplus 1 \precsim 1,$$
(the latter because $1$ is properly infinite). This entails that the $n$-fold direct sum $p \oplus p \oplus \cdots \oplus p$ is properly infinite, whence $p$ is properly infinite by Proposition \ref{Roe}.
\end{proof}

The \Cs{} $\ell^\infty(G) \rtimes_r G$ is also known as the
\emph{Roe-algebra}. More about this algebra in the next section.

\section{Crossed products with exact, non-amenable groups} 
\label{sec:Roe}

It is well-known that the Roe-algebra $\ell^\infty(G)
\rtimes_{r} G$ is \emph{properly
  infinite} precisely when $G$ is non-amenable, and \emph{nuclear} precisely
when $G$ is exact. (The action of $G$ on $\ell^\infty(G)$ is induced
by left-translation.)  
The former is due to the fact that non-amenable
groups are paradoxical (see e.g.\ \cite{Wag:B-T}). See \cite[Theorem
5.1.6]{BroOza} for the latter. In other words, each exact non-amenable
discrete group $G$ admits an amenable action on the compact
(non-metrizable) Hausdorff space 
$\beta G$ so that the corresponding crossed product is
properly infinite and nuclear. 

We show in this section---using these facts---that
every countable discrete exact non-amenable group $G$ admits a free,
minimal, amenable action on the Cantor set $X$ such that the crossed
product $C^*$-algebra $C(X) \rtimes_r G$ is simple, nuclear, and purely
infinite. A related result was obtained by Hjorth and Molberg, \cite{HM}, who proved that any countable discrete group has a free minimal action on the Cantor set admitting an invariant Borel probability measure.  The actions constructed by Hjorth and Molberg give rise to crossed product \Cs s with a tracial state (coming from the invariant measure). Our construction goes in the opposite direction of producing crossed product \Cs s that are traceless, and even purely infinite. 

We note first that the action of $G$ on $\beta G$ given by
left-translation is free. We shall need a stronger result which will
allow us to show that $G$ acts freely also on the spectrum of certain
(separable) sub-$C^*$-algebras of $\ell^\infty(G)$. The
following three (easy) lemmas are devoted to this. It was pointed out to us by David Kerr that our Lemma \ref{lm-free1}  below (and its application to proving freeness) was used back in 1960 by Ellis in \cite{Ellis}, where he proved freeness (strong effectiveness) of any group acting on its universal minimal set. Nonetheless, for the convenience of the reader we have kept the short arguments leading to Lemma \ref{lm-free4} below.

\begin{lemma} \label{lm-free1}
Let $G$ be a group and let $t \in G$, $t \ne e$, be
  given. Then $G$ can be partitioned into pairwise disjoint subsets $G
  = G_1 \cup G_2 \cup G_3$ such that $G_j \cap tG_j =
  \emptyset$ for all $j=1,2,3$.
\end{lemma}

\begin{proof} Let $H$ be a maximal subset of $G$ such that 
$H \cap tH = \emptyset$.
Put
$$G_1 = H, \quad G_2 = tH, \quad G_3 = G
\setminus (G_1 \cup G_2).$$
The fact that $H$ and $tH$ are disjoint implies that $t^k H \cap
t^{k+1}H = \emptyset$ for every integer $k$, and hence in particular
that $t H \cap t^{2}H = \emptyset$ and $t^{-1}H \cap H =
\emptyset$. It remains to 
prove that $G_3 \cap tG_3 = \emptyset$. For this it suffices to show that $G_3
\subseteq t^{-1}H$, or, equivalently, that $G = H \cup
tH \cup t^{-1}H$. Suppose that $s$ 
belongs to $G$ and not to $H$. Then, by maximality of $H$, 
$$\big(H \cup \{s\}\big) \cap t\big(H \cup \{s\}\big) \ne \emptyset.$$ 
As $ts \ne s$, this entails that $s \in tH$ or $ts
\in H$ (or both). Thus either $s \in tH$ or $s \in t^{-1}H$. 
\end{proof}

\noindent
 One can give a perhaps more natural proof of the lemma above by selecting a transversal in $G$ to the right-cosets of the cyclic group generated by $t$. In this way one can show that if $t$ does not have odd order, then $G$ can be partition into two disjoint subsets $G = G_1 \cup G_2$ such that $G_2 = tG_1$. If $t$ has finite odd order, then we need three sets in the partition of $G$ to reach the conclusion of Lemma \ref{lm-free1}.

\begin{corollary} \label{cor-free2}
Let $G$ be a discrete group and let $t \in G$, $t \ne e$, be
  given. Then there are projections $f_1,f_2,f_3$ in $\ell^\infty(G)$
  such that $1 = f_1+f_2+f_3$ and such that $t.f_j \perp f_j$
  for $j=1,2,3$. 
\end{corollary}

\begin{proof} Let $f_j = 1_{G_j}$ where $G_1,G_2,G_3$ are as in the
  previous lemma.
\end{proof}

\noindent The following lemma is trivial, and we omit the proof.

\begin{lemma} \label{lm-free3}
Let $G$ be a discrete group acting on a compact
  Hausdorff space $X$. Suppose that for each $t \in G$, $t \ne e$, one
  can find projections $f_1,f_2,f_3$ in $C(X)$ such that
  $1=f_1+f_2+f_3$ and $t.f_j \perp f_j$ for $j=1,2,3$. 
Then the action of $G$ on $X$ is free.
\end{lemma}

\begin{lemma} \label{lm-free4}
Let $G$ be a countable discrete group. There is a countable subset $M$
 of $\ell^\infty(G)$ such that if $A$
  is any $G$-invariant sub-\Cs{} of $\ell^\infty(G)$ that contains
  $M$, then the action of $G$ on $\widehat{A}$ is free. 
\end{lemma}

\begin{proof} For each $t \in G$, $t \ne e$, use
  Corollary~\ref{cor-free2} to
  choose projections $f_{1,t}$, $f_{2,t}$, and $f_{3,t}$ in
  $\ell^\infty(G)$ such 
  that $f_{1,t}+f_{2,t}+f_{3,t}=1$ and $t.f_{j,t} \perp f_{j,t}$
  for $j=1,2,3$. It follows from
  Lemma~\ref{lm-free3} that the action of $G$ on $\widehat{A}$ is free whenever
  $A$ is a $G$-invariant sub-$C^*$-algebra of $\ell^\infty(G)$ that
  contains the 
  countable set
  $M= \{f_{j,t} \mid t
  \in G\setminus \{e\}, \; j=1,2,3\}$.
\end{proof}

\noindent
The next lemma is contained in the book by N.\ Brown and
Ozawa, \cite{BroOza}:

\begin{lemma}[N.\ Brown and Ozawa] 
\label{lm-amenable} Suppose that $G$ is a countable discrete
  exact group. Then there is a countable subset $M'$ of
  $\ell^\infty(G)$ such that whenever $A$
  is a $G$-invariant sub-\Cs{} of $\ell^\infty(G)$ that contains
  $M'$, then the action of $G$ on $\widehat{A}$ is amenable. 
\end{lemma}

\begin{proof} It is well-known (see for example \cite[Theorem
  5.1.6]{BroOza}) that the action of an exact group $G$ on its
  Stone-Cech compactification $\beta G$ induced by left-multiplication
  on $G \subseteq \beta G$ is amenable. Use the definition of amenable
  actions given in \cite[Definition 4.3.1]{BroOza}, and let $T_i
  \colon G \to \ell^\infty(G)$, $i \in \N$, be finitely supported
  positive functions satisfying the conditions (1), (2), and (3) of
  \cite[Definition 4.3.1]{BroOza}. Then $M' = \{T_i(t) \mid i \in \N,
  \; t \in G\}$ is a 
  countable set with the required properties.
\end{proof}

\begin{lemma} \label{prop.inf.sep.algebra}
Let $G$ be a discrete 
  non-amenable group, and let $p$ be a projection in
  $\ell^\infty(G)$ which is properly infinite in $\ell^\infty(G)
  \rtimes_r G$. 
  Then there is a countable subset $M''$ of
  $\ell^\infty(G)$ such that whenever $A$ is a $G$-invariant unital
  sub-$C^*$-algebra of $\ell^\infty(G)$ which contains $\{p\} \cup M''$,
  then $p$ is properly infinite in $A
  \rtimes_{r} G$.
\end{lemma}

\begin{proof}
There exist partial isometries $x$ and $y$ in $\ell^\infty(G)
\rtimes_{r} G$ satisfying $x^*x=y^*y=p$ and $xx^*+yy^* \le p$. 
For each
$n$ choose $x_n,y_n$ in $C_c(G, \ell^\infty(G))$ such that $\|x-x_n\| <
1/n$ and $\|y-y_n\| < 1/n$. Let $F_n$ be the finite subset of
$\ell^\infty(G)$ consisting of all elements $E(x_n u_t^*)$ and $E(y_n
u_t^*)$ with $t \in G$ (only finitely many of these are non-zero by
the assumption that $x_n$ and $y_n$ belong to $C_c(G,
\ell^\infty(G))$). Then $x_n$ and $y_n$ belong to $A
\rtimes_{r} G$ whenever $A$ is a $G$-invariant
sub-$C^*$-algebra of $\ell^\infty(G)$ that contains $F_n$. Put $M'' =
\cup_{n=1}^\infty F_n$. It follows that whenever $A$ is a $G$-invariant
sub-$C^*$-algebra of $\ell^\infty(G)$ that contains $\{p\} \cup M''$, then
$x$ and $y$ belong to $A \rtimes_{r} G$, whence $p$ is properly
infinite in $A \rtimes_{r} G$.
\end{proof}

\begin{lemma} \label{lm:generators}
Let $G$ be a countable discrete group and let $T$ be a countable subset of
$\ell^\infty(G)$. Then there is a countable and $G$-invariant set
$P$ consisting of projections in $\ell^\infty(G)$ such that $T
\subseteq C^*(P)$. 
\end{lemma}

\begin{proof} Each element in $\ell^\infty(G)$ can be approximated in norm
by elements with finite spectrum, and the $C^*$-algebra generated by
any normal element with finite spectrum is equal to the $C^*$-algebra
generated by finitely many projections (as many projections as there
are elements in the spectrum). It follows that each element in
$\ell^\infty(G)$ belongs to the $C^*$-algebra generated by a countable
set of projections in $\ell^\infty(G)$. Hence the countable set $T$ is
contained in the $C^*$-algebra generated by a countable set $P_0$ of
projections from $\ell^\infty(G)$. The set $P=\bigcup_{t \in G}
t.P_0$ has the desired properties. 
\end{proof}

\begin{proposition} \label{Roe-2}
Let $G$ be a countable discrete group, and let 
$N$ be a countable subset of $\ell^\infty(G)$. Then there exists a
separable $G$-invariant sub-\Cs{} $A$ of $\ell^\infty(G)$ which is generated by
projections and contains $N$, 
and which has the following property: for every
projection $p$ in $A$, if $p$ is properly infinite in $\ell^\infty(G)
\rtimes_r G$, then $p$ is properly infinite in $A \rtimes_r G$.
\end{proposition}

\begin{proof} Let $\cP_\mathrm{inf}$ denote the set of properly infinite
  projections in $\ell^\infty(G) \rtimes_r G$. 
  Use Lemma~\ref{lm:generators} to find a countable
  $G$-invariant set of projections $P_0 \subseteq \ell^\infty(G)$ 
  such that $N \subseteq C^*(P_0)$. Let $Q_0$ be the set of
  projections in $C^*(P_0)$. The set $Q_0$ is countable because
  $C^*(P_0)$ is separable and abelian. For each $p \in Q_0 \cap
  \cP_\mathrm{inf}$ use Lemma~\ref{prop.inf.sep.algebra} to find a
  countable subset 
  $M(p)$ of $\ell^\infty(G)$ such that $p$ is properly infinite in $A
  \rtimes_r G$ whenever $A$ is a $G$-invariant sub-\Cs{} of
  $\ell^\infty(G)$ that contains $\{p\} \cup M(p)$. Put 
$$N_1 = Q_0 \, \cup \, \bigcup_{p \in Q_0 \cap
  \cP_\mathrm{inf}} M(p).$$ 
Use Lemma~\ref{lm:generators} to find a countable
  $G$-invariant set of projections $P_1 \subseteq \ell^\infty(G)$ 
  such that $N_1 \subseteq C^*(P_1)$. 

Continue in this way to find countable subsets $
N_0=N, N_1,N_2, \dots$ of
$\ell^\infty(G)$ and countable $G$-invariant subsets $P_0,P_1,P_2,
\dots$ consisting 
of projections in $\ell^\infty(G)$ such that if $Q_j$ is the
(countable) set of projections in $C^*(P_j)$, then 
$Q_{j} \subseteq N_{j+1}
\subseteq C^*(P_{j+1})$ and every $p \in
Q_{j} \cap \cP_\mathrm{inf}$ is properly infinite in $C^*(P_{j+1})
\rtimes_r G$.

Put $P = \bigcup_{n=0}^\infty P_n$ and put $A = C^*(P)$. Notice that
$$A = \overline{\bigcup_{n=1}^\infty C^*(P_n)}.$$

Then $P$ is a countable $G$-invariant subset of $\ell^\infty(G)$
consisting of projections, and $N \subseteq C^*(P)$. Moreover, if $p$ is a
projection in $A$ which is properly infinite in $\ell^\infty(G)
\rtimes_r G$, then $p$ is equivalent (and hence equal) to a
projection in $C^*(P_n)$ for some $n$, whence $p$ belongs to $Q_n \cap
\cP_{\mathrm{inf}}$, which by construction implies that $p$
 is properly infinite
in $C^*(P_{n+1}) \rtimes_r G$ and hence also in $A \rtimes_r G$.
\end{proof}

An element in a
\Cs{} is said to be \emph{full} if it is not contained in a proper
ideal of the \Cs.

\begin{lemma} \label{Roe-3} Let $A$ be as in
  Proposition~\ref{Roe-2}, and suppose that the group $G$ is
  non-amenable. Then every projection in $A$, which is full in $A
  \rtimes_r G$, is properly
  infinite in $A
  \rtimes_r G$. 
\end{lemma}

\begin{proof} If $p \in A \subseteq \ell^\infty(G)$ is a projection which is full in $A \rtimes_r G$, then $p$ is full in $\ell^\infty(G) \rtimes_r G$, which by Corollary \ref{cor:Roe} implies that $p$ is properly infinite in $\ell^\infty(G) \rtimes_r G$, and hence also properly infinite in $A \rtimes_r G$ by construction of $A$.
\end{proof}

 Let $G$ be a countable discrete group acting on a compact Hausdorff
 space $X$. An open subset $U$ of $X$ is said to be \emph{$G$-full} in
 $X$ if $\bigcup_{t \in G} t.U = X$. If $U$ is a clopen subset of $X$,
 then $U$ is $G$-full in $X$ if and only if $1_U$ is full in $C(X)
 \rtimes_r G$. 

 \begin{lemma} \label{lm:G-full}
 Let $G$ be a countable discrete group acting on a compact, metrizable,
 totally disconnected Hausdorff space $Y$. Let $X$ be a closed $G$-invariant
 subset of $Y$, and let $U$ be a subset of $X$ which is clopen and
 $G$-full relative to $X$. Then there exists a clopen $G$-full subset
 $V$ of $Y$ such that $V \cap X = U$.
 \end{lemma}

 \begin{proof} Let $V_0$ be any clopen subset of $Y$ such that $V_0
   \cap X = U$. Let $(Z_n)$ be an increasing sequence of clopen subsets
   of $Y$ such that $\bigcup_{n=1}^\infty Z_n = Y \setminus X$, and put
   $V_n = V_0 \cup Z_n$. Then each $V_n$ is clopen and $V_n \cap X =
   U$. 

 Since $U$ is a relatively open and $G$-full subset of the compact set
 $X$ there are finitely many elements $t_1,t_2, \dots, t_k \in G$ such
 that $\bigcup_{j=1}^k t_j.U = X$. Put $t_0 = e$ and put $R_n =
 \bigcup_{j=0}^k t_j.V_n$. Then each $R_n$ is clopen and
 $\bigcup_{n=1}^\infty R_n = Y$. By compactness of $Y$ there is $n_0$
 such that $R_{n_0} = Y$. Hence $V = V_{n_0}$ is $G$-full in $Y$, and
 we have already noted that $V$ is clopen and $V \cap X = U$.
 \end{proof}

\noindent A compact Hausdorff space $X$ is homeomorphic to the Cantor
set if and only if it is separable, 
metrizable, totally disconnected, and without
isolated points. Equivalently, the spectrum, $\widehat{A}$, of an abelian
unital $C^*$-algebra $A$  is the Cantor set if and only if $A$ is
separable, is generated as a $C^*$-algebra by its projections, and has
no minimal projections. 

\begin{theorem} \label{thm:crossed-products}
Let $G$ be a countable discrete group. Then $G$ admits a free,
amenable, minimal action on the Cantor set $X$ such that $C(X) \rtimes_r
G$ is a Kirchberg algebra\footnote{A Kirchberg algebra is a simple,
  separable, nuclear, purely infinite \Cs.} in the UCT class if and only if
$G$ is exact and non-amenable.
\end{theorem}

\noindent
We remind the reader that the full and the reduced crossed product
coincides for amenable actions, see e.g.\ \cite[Theorem 4.3.4]{BroOza},
so it does not matter which crossed product we choose in the theorem above.

\begin{proof} The ``only if'' part is well-known: If $G$ is an
  amenable group acting on a unital $C^*$-algebra $A$ that admits a
  tracial state, then the crossed product $A \rtimes G$ also admits a
  tracial state (and hence is not purely infinite). The
  $C^*$-algebra  $C^*_r(G)$ can be embedded into $C(X)
  \rtimes_{r} G$ and is therefore exact if  $C(X)
  \rtimes_{r} G$ is nuclear. Hence $G$ is exact.

The ``if'' part: 
Let $M$ and $M'$ be as in Lemmas~\ref{lm-free4} and \ref{lm-amenable},
respectively. By Proposition~\ref{Roe-2} and Lemma~\ref{Roe-3} there
is a separable 
$G$-invariant sub-\Cs{} $A$ of $\ell^\infty(G)$, which is generated by
projections, which contains $M \cup M' \cup \{1\}$, and which moreover
has the property that every projection in $A$ that is full in $A
\rtimes_r G$ is properly infinite in $A \rtimes_r G$. It follows from
Lemmas~\ref{lm-free4} and \ref{lm-amenable} that the action of $G$ on
$A$ is amenable and the action of $G$ on $Y:= \widehat{A}$ is free. 

The set $Y$ is compact, Hausdorff, metrizable and totally disconnected
(the latter because $A$ is generated by projections). 
The action of $G$ on $Y$ is free and amenable. However,
$Y$ may have isolated points, the action of $G$ on $Y$ may not be
minimal, and $A \rtimes_r G$ need not be purely infinite.

It follows from \cite[Theorem 1.16]{Sier} that the correspondence $I
\mapsto I 
\cap A$ from the set of (closed two-sided) ideals in  $A
\rtimes_{r} G$ to the set of $G$-invariant ideals in $A$ is
bijective. Let $I$ be a maximal proper ideal in $A
\rtimes_{r} G$ and put $J= A \cap I$. Then
$$\big(A \rtimes_{r} G \big)/I \cong (A/J) \rtimes_{r} G
\cong C(X) \rtimes_r G,$$
where $X = \widehat{A/J}$. We can identify $X$ with a closed
$G$-invariant subset of $Y$, and the action of $G$ on $X$ is the
restriction to $X$ of the action of $G$ on $Y$.  

As $G$ acts freely and amenably on $Y$ it also acts freely and
amenably on $X$. (To see that $G$ acts amenably on $X$, or on $A/J =
C(X)$, use the characterization of amenable actions given in
\cite[Definition 4.3.1]{BroOza}. If $T_i \colon G \to A$ satisfy (1),
(2), and (3) in that definition, then $\widetilde{T}_i \colon G \to
A/J$ also satisfy these conditions, when $\widetilde{T}_i = \pi \circ
T_i$ and $\pi \colon A \to A/J$ is the quotient mapping.)

We have now established a free and amenable action of $G$ on a separable
metrizable compact Hausdorff space $X$, and $C(X)=A/J$ is generated by its
projections (because $A$ is generated by its projections). The crossed
product $C(X) \rtimes_r G$ 
is simple by construction which entails that $G$
acts minimally on $X$.  

We show next that $X$ has no isolated points (which will
imply that $X$ is the Cantor set). Assume that $x_0$ is an isolated
point in $X$. Then $G.x_0$ is an open and $G$-invariant subset of
$X$. By minimality of the action of $G$ on $X$ this forces $X =
G.x_0$. As $X$ also is compact this can only happen if $X$ is finite,
which would entail that $G$ is finite ($G$ acts freely), but $G$ is
non-amenable. 

Each non-zero projection in $C(X)$ is full in $C(X) \rtimes_r G$ (by 
simplicity) and hence lifts to a projection in $A$ which is full in $A
\rtimes_r G$ by Lemma~\ref{lm:G-full}. The lifted projection is properly
infinite in $A \rtimes_r G$ by construction (Lemma~\ref{Roe-3}), whence the
given projection in $C(X)$ is properly infinite in $C(X) \rtimes_r
G$. Theorem~\ref{main.theorem.abelian} now implies that $C(X)
\rtimes_r G$ is purely infinite.

The crossed product $C(X) \rtimes_r G$ is the reduced \Cs{} of the 
amenable groupoid $X \rtimes G$, and
those have been proved to belong to the UCT class by Tu in \cite{Tu}.
\end{proof}

We have
not calculated the $K$-theory of the crossed product \Cs{} in
Theorem~\ref{thm:crossed-products}. It ought to
depend on which separable sub-\Cs{} $A$ of $\ell^\infty(G)$
we choose and of which maximal ideal in $A \rtimes_r G$ we divide out
by in the construction. It seems plausible that one should be able to
design the construction such that the $K$-theory of the crossed
product in Theorem~\ref{thm:crossed-products} becomes trivial, in
which case it would be 
isomorphic to $\cO_2$ by the Kirchberg--Phillips classification
theorem.

\bibliographystyle{amsplain}%alphanum} %alternatives: amsalpha, amsplain, alphanu
\providecommand{\bysame}{\leavevmode\hbox to3em{\hrulefill}\thinspace}
\providecommand{\MR}{\relax\ifhmode\unskip\space\fi MR }
\providecommand{\MRhref}[2]{%
  \href{http://www.ams.org/mathscinet-getitem?mr=#1}{#2}
}
\providecommand{\href}[2]{#2}

%\printindex
\end{document}